\documentclass[a4paper,11pt]{article}
\usepackage{amssymb}
\usepackage{amsmath}
\usepackage{amsthm}
\usepackage{hyperref}
\usepackage{enumerate}
\usepackage{xcolor,url}
\usepackage[utf8]{inputenc}

\newtheorem{Theorem} {Theorem} [section]

\newtheorem{Lemma} [Theorem] {Lemma}

\newcommand{\cQ}{{\mathcal Q}}

\newcommand{\PG}{\mathrm{PG}}

\newcommand{\<}{\langle}
\renewcommand{\>}{\rangle} % was: tabbing command
\renewcommand{\phi}{\varphi} % was: phi

\title{Bijection between point-hyperplane antiflags of $V(n, 2)$ and nonsingular points of $O^+(2n, 2)$}
\author{\renewcommand\thefootnote{\alph{footnote}}
Ferdinand Ihringer\footnotemark[1] \and
\renewcommand\thefootnote{\alph{footnote}}
Antonio Pasini\footnotemark[2]
}

\date{11 April 2026}

\begin{document}
\maketitle

{\renewcommand\thefootnote{\alph{footnote}}
\footnotetext[1]{Dept.~of Mathematics,
Southern University of Science and Technology, Shenzhen, Guangdong, China.
E-mail: {\tt ihringer@sustech.edu.cn}}

\footnotetext[2]{Dep.~Information Engineering and Mathematics,
University of Siena, Via Roma 56, I-53100 Siena, Italy.
E-mail: {\tt antonio.pasini@unisi.it}}

\begin{abstract}
  We give a bijection between the point-hyperplane antiflags
  of $V(n, 2)$ and the nonsingular points of $V(2n, \allowbreak 2)$ with respect to a hyperbolic quadric.
  With the help of this bijection, we give a description of the strongly regular graph
  $NO^+_{2n}(2)$ in $V(2n, 2)$. We also describe a graph with respect to a hyperbolic quadric in $V(2n, 2)$
  that was recently defined by Stanley and Takeda in $V(n, 2)$.
  Similarly, we give a bijection between the point-hyperplane antiflags
  of $V(n, 3)$ and the nonsingular points of one type in $V(2n, 3)$ with respect to a hyperbolic quadric.
\end{abstract}

\section{Introduction}

Having different models for the same mathematical object is important.
The most famous example in geometry is the Klein correspondence,
which describes an isomorphism between the lines in the projective
space of dimension three and the singular points of the Klein quadric \cite{Klein}.

The first correspondence, which we describe, is between point-hyperplane
antiflags in $\PG(n-1, 2)$, the projective space of dimension $n-1$
over the field with two elements, and the nonsingular points of $\PG(2n-1, 2)$
with respect to a hyperbolic quadric.
This does not only explain why the number of antiflags and the number of
nonsingular points both equal $(2^n-1) \cdot 2^{n-1}$, but
it also gives a new description for two graphs:
Firstly, it allows us to define the well-known strongly regular graph $NO^+_{2n}(2)$
in the vector space $V(n, 2)$ \cite[p. 85]{BvM}.
Secondly,
we can describe a graph on point-hyperplane
antiflags of $\PG(n-1, 2)$ on the vertex set of $NO^+_{2n}(2)$.
The latter graph has been recently investigated by Stanley and Takeda \cite{StanleyTakeda}.

The second correspondence, which we describe, is between point-hyper\-plane antiflags in $\PG(n-1, 3)$
and the nonsingular points of $\PG(2n-1, 3)$ with respect to a hyperbolic quadric of one type.
Both numbers equal $\frac12 (3^n-1) \cdot 3^{n-1}$. As a consequence, we find a description
of the strongly regular graph $NO^+_{2n}(3)$ in $\PG(n-1, 3)$.

\section{Preliminaries}

We work in $V(n, q)$, the vector space of dimension $n$
over the (finite) field with $q$ elements. 
We use vector space notation and algebraic dimensions,
while using some geometric wording: call $1$-spaces
\textit{points}, $2$-spaces \textit{lines}, $3$-spaces \textit{planes},
and $(n-1)$-spaces \textit{hyperplanes}. We say that two subspaces
are disjoint if their intersection is the trivial subspace.
A \textit{point-hyperplane antiflag} is a pair $(P, H)$,
where $P$ is a point and $H$ is a hyperplane with $P \nsubseteq H$.

Two distinct point-hyperplane antiflags $(P, H)$ and $(P', H')$ of $V(n, 2)$
are in one of the following relations:
\begin{itemize}
 \item Relation $A_1$: $P \subseteq H'$ and $P' \nsubseteq H$ or $P \nsubseteq H'$ and $P' \subseteq H$.
 Degree: $(2^{n-1}-1) \cdot 2^{n-1}$.
 \item Relation $A_2$: $P \subseteq H'$ and $P' \subseteq H$. Degree: $(2^{n-1}-1) \cdot 2^{n-2}$.
 \item Relation $A_3$: Either $P = P'$ and $H \neq H'$ or $P \neq P'$ and $H = H'$.
 Degree: $2(2^{n-1}-1)$.
 \item Relation $A_4$: $P \nsubseteq H'$, $P' \nsubseteq H$, $P \neq P'$, and $H \neq H'$.
 Degree: $(2^{n-1}-1)(2^{n-2}-1)$.
\end{itemize}
The graph with the anti-flags as set of vertices, two adjacent if in relation $A_2$,
has been recently investigated Stanley and Takeda \cite[\S4]{StanleyTakeda}.

For $V(n, q)$, $q \geq 3$, Relation $A_4$ splits into two and we replace it by
the following two relations $A_4'$ and $A_4^*$:
\begin{itemize}
 \item Relation $A_4'$: $P \nsubseteq H'$, $P' \nsubseteq H$, $P \neq P'$, $H \neq H'$, and $P+P'$ meets $H \cap H'$.
 \item Relation $A_4^*$: $P \nsubseteq H'$, $P' \nsubseteq H$, $P \neq P'$, $H \neq H'$, and $P+P'$ does not meet $H \cap H'$.
\end{itemize}

Equip a vector space $V(m, q)$ with a quadratic form $Q(x) = x^T A x$.
Say that $Q$ is \textit{nondegenerate} if $A$ has full rank.
Call a point $\< x \>$ \textit{singular} if $Q(x) = 0$.
Call a subspace $S$ \textit{totally singular} if all its points are singular.
The largest dimension of a totally singular subspace is called the \textit{rank} of $Q$.
There are the following nondegenerate quadratic forms up to isomorphy (e.g. see \cite{Taylor91}):
\begin{itemize}
 \item For $m$ even, $Q(x) = x_1 x_2 + x_3 x_4 + \cdots + x_{m-1} x_{m}$
        with its associated geometry of (totally) singular subspaces
        denoted by $O^+(m, q)$, a \textit{hyperbolic quadric}.
        The rank is $m/2$.
 \item For $m$ even, $Q(x) = \alpha x_1^2 + \beta x_1 x_2 + x_2^2 + x_3 x_4 + \cdots + x_{m-1} x_{m}$,
        where $\alpha x_1^2 + \beta x_1x_2 + x_2^2$ is irreducible,
        with its associated geometry of (totally) singular subspaces
        denoted by $O^-(m, q)$, an \textit{elliptic quadric}.
        The rank is $m/2-1$.
 \item For $m$ odd, $Q(x) = x_1^2 + x_2 x_3 + \cdots + x_{m-1} x_m$
        with its associated geometry of (totally) singular subspaces
        denoted by $O(m, q)$, a \textit{parabolic quadric}.
        The rank is $(m-1)/2$.
\end{itemize}

Now suppose that $m=2n$, $q=2$, and $Q$ is hyperbolic, so we work
with $O^+(2n, 2)$. Let $\cQ$ denote the set of singular points of $O^+(2n, 2)$.
We will use that in $\cQ$ a totally isotropic subspace of dimension $n-1$ lies in precisely two maximal totally isotropic subspaces.
Let $\perp$ denote the orthogonality relation of the bilinear form associated with the quadratic form $Q$ of $O^+(2n, 2)$.
For $q=2$, we only find two nontrivial relations
between nonsingular points $X := \< x \>$ and $X' := \< x' \>$:
\begin{itemize}
 \item Relation $B_1$: $\< x+x'\>$ is nonsingular, that is,
 the line $X+X'$ does not meet $\cQ$. Degree: $(2^{n-1}-1) \cdot 2^{n-1}$.
 \item Relation $B_2$: $\< x+x'\>$ is singular, that is, the line $X+X'$
 is tangent to $\cQ$. Degree: $(2^{n-1}+1)(2^{n-1}-1)$.
\end{itemize}
A \textit{strongly regular graph} with parameters $(v, k, \lambda, \mu)$
is a graph (not complete, not edgeless) with $v$ vertices, each vertex has degree $k$, two adjacent vertices
have precisely $\lambda$ common neighbors, and two non-adjacent vertices have
precisely $\mu$ common neighbors.
Both relations $B_1$ and $B_2$ define well-known strongly regular graphs, see \cite[\S3.1.2]{BvM}.
The graph with relation $B_2$ is usually denoted by $NO^+_{2n}(2)$.
Its complement, the graph with relation $B_1$, has parameters
\begin{align*}
  & v= (2^n-1) \cdot 2^{n-1}, && k = (2^{n-1}-1) \cdot 2^{n-1},\\
  & \lambda = (2^{n-2}-1) \cdot 2^{n-1}, && \mu = (2^{n-1}-1) \cdot 2^{n-2}.
\end{align*}

Changing $q=2$ to $q=3$, then there are two types of nonsingular points $\< x \>$:
those with $Q(x) = 1$ and those with $Q(x) = -1$. For two nonsingular points $X := \< x \>$
and $X' := \< x' \>$ with $Q(x) = Q(x') = 1$, we only find two nontrivial solutions:
\begin{itemize}
 \item Relation $B_1^*$: $Q(x+x') = -1$, that is,
 the line $X+X'$ does not meet $\cQ$.
 \item Relation $B_2^*$: $Q(x+x') \neq -1$, that is, the line $X+X'$
 is tangent to $\cQ$.
\end{itemize}
Again, both relations $B_1^*$ and $B_2^*$ define well-known strongly regular graphs, see \cite[\S3.1.3]{BvM}.
The graph with relation $B_1^*$ is usually denoted by $NO^+_{2n}(3)$
and has parameters
\begin{align*}
 & v = \frac12 (3^m-1) \cdot 3^{m-1}, && k = \frac12 (3^{m-1}-1) \cdot 3^{m-1},\\
 & \lambda = \frac12 (3^{m-1}+1) \cdot 3^{m-2}, && \mu = \frac12 (3^{m-2}+1) \cdot 3^{m-1}.
\end{align*}
For $q=3$, we will make use of the well-known fact that if $X+X'$ does not meet
$\cQ$, then (for $X, X'$ of the same type) $X \perp X'$.

% We will use the following facts about $\cQ$:
% \begin{itemize}
%  \item For a nonsingular point $X$, the induced quadratic form on $X^\perp$ is isomorphic to a parabolic quadric $O(2n-1, 2)$.
%  \item For a singular point $Y$, the induced quadratic form on $Y^\perp$ is isomorphic to a cone
%         with $Y$ as vertex and a hyperbolic quadric $O^+(2n-2, 2)$ as base.
%  \item For a tangent line $L$ with a singular point $Y$,
%         the induced quadratic form on $L^\perp$ is isomorphic to a cone  with $Y$ as vertex
%         and a parabolic quadric $O(2n-3, 2)$ as base.
%  \item A totally isotropic subspace of dimension $n-1$ lies in precisely two maximal totally isotropic subspaces.
% \end{itemize}
% Now we are ready to show our results.

\section{The First Correspondence}

% We work with $V(n, 2)$ and $O^+(2n, 2)$ (naturally embedded) in $V(2n, 2)$.
As before, $\cQ$ denotes the singular points of $O^+(2n, 2)$
(naturally embedded in $V(2n, 2)$).
Within $\cQ$, fix two disjoint maximal singular subspaces
$\Pi$ and $\Sigma$. Our goal is to describe a bijection $f$
between the nonsingular points in $V(2n, 2)$ and
the point-hyperplane antiflags in $V(n, 2)$.
%
% For a nonsingular point $X$ of $V(2n, 2)$, the hyperplane $X^\perp$
% intersects $\cQ$ in a parabolic quadric isomorphic with $O(2n-1, 2)$.
% Thus, the subspaces $G := X^\perp \cap \Pi$ and $H := X^\perp \cap \Sigma$
% have dimension $n-1$ each. Furthermore, using $\dim(G^\perp) = 2n-(n-1) = n+1$
% and $\dim(\Sigma) = n$, we see that $G^\perp \cap \Sigma$
% is a point $P$. We define $f$ by $f(X) = (P, H)$.
% Now we claim that we can identify the $n$-space $\Sigma$ with $V(n, 2)$
% and that $f$ is a desired bijection.
% Several of the following observations are well-known,
% for instance, see Proposition 5.2 in \cite{VMV19}.
% We include proofs to make this note self-contained.

Every point $P$ of $V(2n, 2)$ exterior to $\Pi$ and $\Sigma$
defines a unique point $P_\Pi$ on $\Pi$ by
$P_\Pi = (P+\Sigma) \cap \Pi$, and
a unique point $P_{\Sigma}$ on $\Sigma$ by $P_{\Sigma} = (P+\Pi) \cap \Sigma$.
Note that $P$, $P_\Pi$, and $P_\Sigma$ are collinear.
Indeed, by definition, the line joining $P$ with $P_\Pi$ is contained in $P+\Sigma$, hence it meets $\Sigma$
in a point which, by definition of $P_\Sigma$, is necessarily equal to $P_\Sigma$.
Each point $X$ of $\Sigma$ defines a hyperplane $H(X)$ of $\Pi$ by $H(X) = X^\perp \cap \Pi$.
Note that
 \begin{align*}
    H(P_\Sigma) &= H((P+\Pi) \cap \Sigma) = ((P+\Pi) \cap \Sigma)^\perp \cap \Pi\\
    &= ((P^\perp \cap \Pi) + \Sigma) \cap \Pi = P^\perp \cap \Pi.
 \end{align*}
Define a function $f$ from the points of $V(2n, 2)$ exterior to $\Pi$ and $\Sigma$ to
point-hyperplane pairs of $\Pi$ by $f(P) = (P_\Pi, H(P_\Sigma))$.

The following lemmata are well-known in more generality,
for instance, see Proposition 5.2 in \cite{VMV19}.
We include proofs to make this note self-contained.

\begin{Lemma}\label{lem:uniqueln}
 If $P$ is a point exterior to $\Pi$ and $\Sigma$,
 then $P$ belongs to a unique line $L = \{ P, P_\Pi, P_\Sigma \}$.
 In particular, if $P$ is in $\cQ$, then $L$ is totally singular.
\end{Lemma}
\begin{proof}
Any point $P$ exterior to $\Pi$ and $\Sigma$ lies in a unique line $L$ crossing both $\Pi$ and $\Sigma$
as $V(2n, q)$ is the direct sum of $\Pi$ and $\Sigma$.
By definition, $L$ contains the points $P$, $P_\Pi$, and $P_\Sigma$.
\end{proof}

\begin{Lemma}\label{lem:antiflag}
 For a nonsingular point $P$, the pair $f(P) = (P_\Pi, H(P_\Sigma))$ is an antiflag of $\Sigma$.
\end{Lemma}
\begin{proof}
 We have $H(P_\Sigma) = P^\perp \cap \Pi$. By Lemma \ref{lem:uniqueln},
 $P_\Pi \subseteq P^\perp$ if and only if $P$ is a singular point.
\end{proof}

\begin{Lemma}
 The function $f$ is bijective.
\end{Lemma}
\begin{proof}
 The number of nonsingular points in $V(2n,2)$
 and the number of point-hyperplane antiflags in $V(n, 2)$ are both equal to
 $(2^n-1) \cdot 2^{n-1}$.
 We still need to show that $f$ is injective. Indeed, by Lemma \ref{lem:uniqueln},
 $P_\Pi$ and $P_\Sigma$ (and, therefore, $P_\Pi$ and $H(P_\Sigma)$) uniquely determine $P$.
\end{proof}

Define the following non-trivial relations on the vertices of $NO^+_{2n}(2)$,
that is, the set of nonsingular points of $V(2n, 2)$ with respect to $O^+(2n, 2)$:
\begin{itemize}
 \item Relation $C_1$: $P+P'$ is disjoint from $\cQ$. Degree: $(2^{n-1}-1) \cdot 2^{n-1}$.
 \item Relation $C_2$: $P+P'$ is a tangent and meets $\cQ$ in a point $Y$ exterior to $\Pi$ and $\Sigma$,
 where the unique line through $Y$ crossing both $\Pi$ and $\Sigma$ (see Lemma \ref{lem:uniqueln}) is not contained in $(P+P')^\perp$.
 Degree: $(2^{n-1}-1) \cdot 2^{n-2}$.
 \item Relation $C_3$: $P+P'$ is a tangent and meets $\Pi$ or $\Sigma$. Degree: $2(2^{n-1}-1)$.
 \item Relation $C_4$: $P+P'$ is a tangent and meets $\cQ$ in a point $Y$ exterior to $\Pi$ and $\Sigma$,
 where the unique line through $Y$ crossing both $\Pi$ and $\Sigma$ (see Lemma \ref{lem:uniqueln}) is contained in $(P+P')^\perp$.
 Degree: $(2^{n-1}-1)(2^{n-2}-1)$.
\end{itemize}
   
While these numbers can be counted directly, we will give an isomorphism
between $A_i$ and $C_i$.

\begin{Theorem}\label{thm:corr1}
 Let $i \in \{1,2,3,4\}$.
 Two vertices $P, P'$ of $NO^+_{2n}(2)$ are in relation $C_i$
 if and only if $f(P)$ and $f(P')$ are in relation $A_i$.
\end{Theorem}
\begin{proof}
Let $P$ and $P'$ be distinct nonsingular points.
Let $R$ denote the third point on the line $P+P'$.
In particular, $R$ is in $P^\perp \cap P'^\perp$.
Thus, if $R$ is in $\Pi$, then $P_\Pi = P'_\Pi$
(as $H(P_\Sigma) = P^\perp \cap \Pi$).
Similarly, if $R$ is in $\Sigma$, then $P_\Sigma = P'_\Sigma$.
This corresponds to Case $A_3$ and Case $C_3$.
Now assume that $R$ is exterior to $\Pi$ and $\Sigma$.
Put $L_\Pi = \{ P_{\Pi}, P'_{\Pi}, R_{\Pi}\}$
and $L_{\Sigma} = \{ P_\Sigma, P'_\Sigma, R_\Sigma \}$.
Note that $L_\Pi$ and $L_\Sigma$ are lines as they are the projection
of the line $L = \{ P, P', R \}$ onto $\Pi$, respectively, $\Sigma$.

\smallskip

If $R$ is not in $\cQ$, then $L_\Pi^\perp$ is disjoint from $L_\Sigma$. Thus, either
$P_\Pi \perp P'_\Sigma$, $R_\Pi \perp P_\Sigma$, and $P'_\Pi \perp R_\Sigma$;
or $P_\Pi \perp R_\Sigma$, $P'_\Pi \perp P_\Sigma$, and $R_\Pi \perp P'_\Sigma$.
This corresponds to Case $A_1$ and Case $C_1$.

\smallskip

Suppose that $R$ is in $\cQ$ and that $L_\Pi^\perp$ is disjoint from $L_\Sigma$. Then, similarly to the previous case, $P_\Pi \perp P'_\Sigma$
and $P_\Sigma \perp P'_\Pi$. This corresponds to Case $A_2$ and $C_2$.

\smallskip
Now suppose that $R$ is in $\cQ$, but that $L_\Pi^\perp$ meets $L_\Sigma$.
Then $L_\Sigma \subseteq X^\perp$ for one $X \in \{ P_\Pi, P'_\Pi, R_\Pi \}$.
But $L_\Sigma \nsubseteq P_\Pi^\perp, P_\Pi'^\perp$ by Lemma \ref{lem:uniqueln},
so $X = R_\Pi$. Similarly, $L_\Pi \subseteq R_\Sigma^\perp$.
Thus, this corresponds to Case $A_4$ and Case $C_4$.
\end{proof}

\section{The Second Correspondence}

For our second correspondence, the setup and argument are almost identical. Therefore, we limit ourselves
to highlight the differences in the argument.

For the non-trivial relations on the vertices of $NO^+_{2n}(3)$,
the defintion is identical to the previous correspondence except for $C_1$
which splits into two relations $C_1'$ and $C_1^*$:
\begin{itemize}
 \item Relation $C_1'$: $P+P'$ is disjoint from $\cQ$, $P_\Sigma \subseteq P_\Pi'^\perp$, and $P_\Sigma' \subseteq P_\Pi^\perp$.
 \item Relation $C_1^*$: $P+P'$ is disjoint from $\cQ$, $P_\Sigma \nsubseteq P_\Pi'^\perp$, and $P_\Sigma' \nsubseteq P_\Pi^\perp$.
\end{itemize}

\begin{Theorem}
 Two vertices $P, P'$ of $NO^+_{2n}(3)$ are in relation $N$
 if and only if $f(P)$ and $f(P')$ are in relation $M$ for the following pairs of $M$ and $N$:
 \begin{enumerate}
  \item $M = A_1$ and $N = C_2$,
  \item $M = A_2$ and $N = C_1'$,
  \item $M = A_3$ and $N = C_3$,
  \item $M = A_4'$ and $N = C_4$,
  \item $M = A_4^*$ and $N = C_1^*$.
 \end{enumerate}
\end{Theorem}
\begin{proof}
 Let $\{P, P', R, S \}$ be the points of the line $P+P'$.
 Recall that as $P$ and $P'$ have the same type, so the line through them
 cannot be a secant of $\cQ$.
 If either $R$ or $S$ are in $\Pi$ or $\Sigma$, then we are in
 Case $A_3$ and Case $C_3$.
 Thus, now we assume that neither of them is in $\Pi$ or $\Sigma$.
 Put $L_\Pi = \{ P_\Pi, P'_\Pi, R_\Pi, S_\Pi \}$
 and $L_\Sigma = \{ P_\Sigma, P'_\Sigma, R_\Sigma, S_\Sigma \}$.
 The proof is identical to Theorem \ref{thm:corr1},
 except for when $L_\Pi^\perp$ is disjoint from $L_\Sigma$
 as we now have different options.

 \smallskip

 If, say, $S$ is in $\cQ$, then
 either $P_\Pi \perp P'_\Sigma$, $R_\Pi \perp P_\Sigma$, and $P'_\Pi \perp R_\Sigma$;
or $P_\Pi \perp R_\Sigma$, $P'_\Pi \perp P_\Sigma$, and $R_\Pi \perp P'_\Sigma$.
This corresponds to Case $A_1$ and Case $C_2$.

\smallskip

If none of $R$ or $S$ are in $\cQ$, then (up to switching $R$/$S$ and $P$/$P'$) one of the following occurs:
\begin{enumerate}
 \item $P_\Pi \perp P'_\Sigma$, $P'_\Pi \perp P_\Sigma$, $R_\Pi \perp S_\Sigma$, $S_\Pi \perp R_\Sigma$;
%  \item $P_\Pi \perp P'_\Sigma$, $P'_\Pi \perp R_\Sigma$, $R_\Pi \perp S_\Sigma$, $S_\Pi \perp P_\Sigma$;
%  \item $P_\Pi \perp R_\Sigma$, $P'_\Pi \perp P_\Sigma$, $R_\Pi \perp S_\Sigma$, $S_\Pi \perp P'_\Sigma$;
 \item $P_\Pi \perp R_\Sigma$, $P'_\Pi \perp S_\Sigma$, $R_\Pi \perp P'_\Sigma$, $S_\Pi \perp P_\Sigma$;
%  \item $P_\Pi \perp R_\Sigma$, $P'_\Pi \perp S_\Sigma$, $R_\Pi \perp P'_\Sigma$, $S_\Pi \perp P_\Sigma$.
\end{enumerate}
The first case corresponds to Case $C_1'$ and Case $A_2$.
The second case corresponds to Case $C_1^*$ and Case $A_4^*$.
\end{proof}

Clearly, $B_1^* = C_1^* \cup C_1'$.
Thus, $A_2 \cup A_4^*$ gives a different model for the strongly regular graph $NO^+_{2n}(3)$.
Similarly, $A_1 \cup A_3 \cup A_4'$ describes $B_2^*$.

\paragraph*{Acknowledgements} We thank the referee for their excellent comments, in particular, for pointing out the second correspondence.

\end{document}